\documentclass[12pt]{amsart}
\usepackage{graphicx}
\usepackage{float}
\usepackage{amsmath,amsthm, amsfonts,amssymb, stmaryrd,yfonts,pxfonts,pifont,eufrak, bbm}
\usepackage{tikz}
\usepackage{caption}
\usepackage{subcaption}
\newtheorem{Cor}{Corollary}
 \newtheorem{Lemma}{Lemma}
 
 \newtheorem{ex}{Example}
 \newtheorem{Proposition}{Proposition}
 \theoremstyle{definition}
 
 \theoremstyle{remark}
 \newtheorem{Remark}[Lemma]{Remark}
 \numberwithin{equation}{subsection}

\begin{document}

\title[TRANSITIVITY AND RELATED NOTIONS FOR GRAPH INDUCED SYMBOLIC SYSTEMS]{TRANSITIVITY AND RELATED NOTIONS FOR GRAPH INDUCED SYMBOLIC SYSTEMS}%
\author{Prashant Kumar and Puneet Sharma}
\address{Department of Mathematics, I.I.T. Jodhpur, NH 65, Nagaur Road, Karwar, Jodhpur-342030, INDIA}%
\email{puneet@iitj.ac.in, kumar.48@iitj.ac.in}%




\begin{abstract}
In this paper, we investigate the dynamical behavior of a two dimensional shift $X_G$ (generated by a two dimensional graph $G=(\mathcal{H},\mathcal{V})$) using the adjacency matrices of the generating graph $G$. In particular, we investigate properties such as transitivity, directional transitivity, weak mixing, directional weak mixing and mixing for the shift space $X_G$. We prove that if $(HV)_{ij}\neq 0 \Leftrightarrow (VH)_{ij}\neq 0$ (for all $i,j$), while doubly transitivity (weak mixing) of $X_H$ (or $X_V$) ensures the same for two dimensional shift generated by the graph $G$, directional transitivity (in the direction $(r,s)$) can be characterized through the block representation of $H^rV^s$. We provide necessary and sufficient criteria to establish horizontal (vertical) transitivity for the shift space $X_G$. We also provide examples to establish the necessity of the conditions imposed. Finally, we investigate the decomposability of a given graph into product of graphs with reduced complexity.
\end{abstract}
\maketitle

\section{Introduction}
Symbolic dynamics has widely been used to investigate the dynamics of a various physical and natural processes around us. As any general dynamical system can be embedded in a symbolic system with appropriate number of symbols \cite{fu}, it is sufficient to study the shift spaces and its subsystems to investigate the dynamics of a general discrete dynamical system. As the study of a general symbolic system can be used to model the dynamics of various systems arising in different fields of sciences and engineering, the area has found applications in various areas of science and engineering such as automata theory, data transmission, data storage and communication systems \cite{bruce,shanon,lind1}.\\

In recent years, multidimensional symbolic dynamics has emerged as a topic of interest and has captivated the attention of several researchers across the globe. The convenience of symbol-based representation allows a system with higher complexity to be studied through multidimensional symbolic systems. The easier computability and wide applicability of such systems has enabled the topic to find applications in different fields of sciences and engineering including mathematics, computer science, physics and many more. In particular, area has found applications areas such as computational biology, tiling spaces and study of quasicrystals \cite{eug,robinson,song,ted}. In recent times, several researchers have extensively investigated various dynamical and structural aspects of multidimensional shifts and interesting results have been obtained. In \cite{ber}, the author established the undecidability of the non-emptiness problem for multidimensional shift spaces. In \cite{pd}, the authors characterized any multidimensional shift of finite type using an infinite matrix. In \cite{pd2}, the authors provide an algorithm for generating the elements of the shift space using a sequence of ﬁnite matrices (of increasing order). The authors establish that the sequence generated yields precisely the elements of the shift space under consideration and hence characterizes the elements of the shift space. In \cite{sam}, the authors established that any strongly irreducible two dimensional shift space has a dense set of periodic points. In \cite{quas}, the authors proved that any multidimensional shifts of finite type exhibiting positive topological entropy cannot be minimal. In particular, they established that any shift of finite type with positive topological entropy contains a shift which is not of finite type and hence contains infinitely many shifts of finite type. In \cite{boy}, the authors investigate mixing $Z^d$ shifts of finite type and sofic shifts with large entropy. They give examples to prove that there exists $Z^d$-mixing systems such that no non-trivial full shift is a factor for such systems. They also establish existence of sofic systems such that the only minimal subsystem is a singleton. In \cite{hoch2}, authors prove that the set of entropies for $Z^d$ shifts of finite type ($d\geq 2$) coincides with the set of the infimum of all possible recursive sequences of rational numbers. In \cite{hoch1}, the author established that any $h\geq 0$ is the entropy of a $Z^d$ effective dynamical system if and only if it is the lim inf of a recursive sequence of rational numbers. In \cite{pp}, authors discuss the non-emptiness problem and the existence of periodic points for a multidimensional shift space. They derive conditions under which the shift space is non-empty and posseses periodic points. The authors investigate the structure of the shift space using the generating matrices. They prove that any $d$-dimensional shift of finite type is finite if and only if it is conjugate to a shift generated through permutation matrices. In \cite{markley}, authors investigated the structure and mixing properties of two-dimensional shift spaces generated by consistent matrices. In \cite{markley1}, the authors discuss the entropy and ergodicity of a two-dimensional shift space using consistent matrices. In \cite{Nat}, authors gave generalized definitions of transitivity for two-dimensional shift spaces and studied various mixing notions for Dot Systems. Before we move further, we provide some of the basic definitions and concepts required. \\

Let $A = \{a_i : i \in I\}$ be a finite set, $d$ be a positive integer and let $A$ be equipped with the discrete metric. Let $A^{\mathbb{Z}^d}$ be equipped with the product topology. Then, the metric $\mathcal{D} : A^{\mathbb{Z}^d} \times A^{\mathbb{Z}^d} \rightarrow \mathbb{R}^+$  defined as $\mathcal{D} (x,y) = \frac{1}{n+1}$  (where $n$ is the least non-negative integer such that $x \neq y$ in $R_n = [-n,n]^d$) generates the product topology. For any $a\in \mathbb{Z}^d$, the map $\sigma^a : A^{\mathbb{Z}^d} \rightarrow A^{\mathbb{Z}^d}$ defined as $(\sigma^a (x))(k)= x(k+a)$ is a $d$-dimensional shift and is a homeomorphism. Any $X \subset A^{\mathbb{Z}^d}$ is called shift invariant if $\sigma^a(X)\subset X$ for all $a\in\mathbb{Z}^d$. Any non-empty, closed, shift invariant subset of $A^{\mathbb{Z}^d}$ is called a \textbf{subshift}. Let $\mathcal{F}$ be a given set of finite patterns and let $X=\overline{\{x\in A^{\mathbb{Z}^d}: \text{any pattern from~~} \mathcal{F} \text{~~does not appear in~~} x \}}$. Then $X$ defines a subshift of $A^{\mathbb{Z}^d}$ generated by set of forbidden patterns $\mathcal{F}$. We say that the shift space $X$ is a \textbf{shift of finite type} if it can be generated by a finite forbidden set of finite patterns. We say that a pattern is \textbf{allowed} if it is not an extension of any forbidden pattern.  See \cite{bruce,coven,lind1,Jan,Hoch,morse,pd} for details.\\

A subshift $X$ is called \textbf{transitive} if for any pair of non-empty open sets $U,V$ in $X$, there exists  $n\in\mathbb{Z}^d$ such that $\sigma^n (U)\cap V\neq \emptyset$. A one-dimensional subshift $X$ is called \textbf{doubly transitive} if both the  left and right shifts are transitive for positive time. For any $r\in\mathbb{Z}^d$, a subshift $X$ is called \textbf{r-transitive} if for any non-empty open sets $U,V$ in $X$, there exists  $k\in\mathbb{Z}$ such that $\sigma^{kr} (U)\cap V\neq \emptyset $. It may be noted that a subshift  $X$ is $r$-transitive if the system $(X,\sigma^r)$ is transitive (as a one dimensional system). A shift space $X$ is called \textbf{totally transitive} if $(X,\sigma^r)$ is transitive for all $r\in\mathbb{Z}^d \setminus \{\textbf{0}\}$. A subshift $X$ is called \textbf{weakly mixing} if for any pairs of non-empty open sets $U_1, U_2$ and $V_1,V_2$ in $X$, there exists  $n\in\mathbb{Z}^d$ such that $\sigma^n (U_i)\cap V_i\neq \emptyset~~$ for $i=1,2$. For any $r\in\mathbb{Z}^d$, a subshift $X$ is called \textbf{ r-weakly mixing} if for any pairs of non-empty open sets $U_1, U_2$ and $V_1,V_2$ in $X$, there exists  $k\in\mathbb{Z}$ such that $\sigma^{kr} (U_i)\cap V_i\neq \emptyset~~$ for $i=1,2$. Once again, it may be noted that a subshift  $X$ is $r$-weakly mixing if the system $(X,\sigma^r)$ is weakly mixing (as a one dimensional system). A subshift $X$ is called \textbf{mixing} if for any non-empty open sets $U$ and $V$ in $X$, the set $\{n\in\mathbb{Z}^d: \sigma^n (U)\cap V\neq \emptyset \}$ is cofinite. Refer \cite{bruce,Hoch,Nat,TK} for details.\\

Let $G=(V,E)$ be a graph (where $V$ and $E$ are the set of vertices and edges respectively). It may be noted that the set of bi-infinite walks over the graph $G$ is a one dimensional (one step) shift of finite type. Also, as any given shift of finite type $X$ can be coded as a higher block shift (conjugate to $X$ and can be generated by a finite graph $G$), every one dimensional shift of finite type can be visualized as a shift generated from some graph. Also, if $\{G_1,G_2,\ldots,G_d\}$ is a set of graphs with a common set of vertices $V$, the collection naturally induces a $d$-dimensional shift of finite type (where $i$-th graph determines the compatibility of the vertices in the $i$-th direction). In such a case, we refer $(G_1,G_2,\ldots,G_d)$ as a \textbf{$d$-dimensional graph}. We say that a two dimensional graph $G=(H, V)$ is \textbf{connected} if for every pair of vertices $i , j \in \mathcal{V}(G)$, $\exists \ (r, s) \in (\mathbb{Z}^{+} \times \mathbb{Z}^{+}) \setminus\{0,0\}$ such that $(H^r V^s)_{ij} >0$. It may be noted that if $(HV)_{ij}\neq 0 \Leftrightarrow (VH)_{ij}\neq 0$ for all $i,j$, then any $1\times r$ strip can extended to a $2\times r$ (and hence to $n\times m$ block for any $n,m\in\mathbb{N}$) block allowed for the shift space. Consequently, the condition ensures extension of a given $r\times s$ block to a valid configuration for the shift space under discussion. See \cite{bruce,lind1,markley,markley1,pp} for details.\\ 

An $n\times n$ real, non-negative, square matrix $A$ is called \textbf{irreducible} if for any pair of indices $i,j\in \{1,2,\ldots,n\}$, there exists $k\in\mathbb{N}$ such that $A^k_{ij}\neq 0$. Such a matrix $A$ is called \textbf{primitive} if there exists $k\in\mathbb{N}$ such that $A^k_{ij}\neq 0~~ \forall i,j\in\{1,2,\ldots,n\}$. It can be seen that a one dimensional shift generated by the matrix $A$ is doubly transitive (strongly mixing) if and only if $A$ is irreducible (primitive). Further, it is known that any one dimensional shift of finite type $X$ is totally transitive $\Leftrightarrow X$ is weakly mixing  $\Leftrightarrow X$ is strongly mixing. Refer \cite{bruce,TK} for details.\\

Let $G_1$ and $G_2$ be two directed graphs. Then, the Cartesian (Tensor) product of $G_1$ and $G_2$ is defined as the graph with  $\mathcal{V}(G_1) \times \mathcal{V}(G_2)$ as the set of vertices. The vertices $(a,b)$ and $(c,d)$ are connected in \textbf{Cartesian Graph Product} (denoted by $G_1  \scalebox{0.7}{$\square$} G_2$), if $ a= c \ (ac \in \mathcal{E}(G_1))$ and $ bd \in \mathcal{E}(G_2) \ (b=d)$. In other words, two tuples are connected in Cartesian graph product if one of the indices (of the tuple) coincide and the other is connected (in their respective graphs). The vertices $(a,b)$ and $(c,d)$ are connected in \textbf{Tensor Graph Product} (denoted by $G_1 \times G_2$) if $ ac \in \mathcal{E}(G_1)$ and $bd \in \mathcal{E}(G_2)$. We define the graph product (Cartesian and Tensor) of two-dimensional graphs $G_1$ and $G_2$ to be \textbf{component-wise graph product}. See \cite{ sabi, paul} for details. \\

In this paper, we investigate the dynamical behavior of a two dimensional shift generated by the graph $G=(\mathcal{H},\mathcal{V})$). For the sake of simplicity, we use $H$ and $V$ for the adjacency matrices of $\mathcal{H}$ and $\mathcal{V}$ respectively. In particular, we investigate transitivity, $r$-transitivity, total transitivity, weakly mixing, $r$-weakly mixing and mixing for shift space under consideration. We prove that if $(HV)_{ij}\neq 0 \Leftrightarrow (VH)_{ij}\neq 0, ~~\forall i,j$, while doubly transitivity (weak mixing) of $X_H$ (or $X_V$) ensures the same for two dimensional shift generated by the graph $G$, irreducibility of $H^rV^s$ provides a necessary and sufficient criteria to establish directional doubly transitivity (in $(r,s)$ direction ($rs>0$)) for the multidimensional shift space under discussion. We provide examples to show that the results derived do not hold if any of the conditions imposed fails to hold. We also provide necessary and sufficient criteria to establish horizontal (vertical) transitivity for the shift space $X_G$. We also investigate the decomposability of a given graph into product of graphs with reduced complexity. \\

\section{Main Results}
\begin{Proposition}
Let $X_{G}$ be a shift space generated by the two dimensional graph  $G=(\mathcal{H},\mathcal{V})$. If $H$ and $V$ are irreducible permutation matrices such that $HV=VH$, then $X_{G}$ comprises of a single periodic point (with finite orbit).
\end{Proposition}

\begin{proof}
Let $X_{G}$ be shift space generated by the two dimensional graph $G=(\mathcal{H},\mathcal{V})$. Firstly, note that if $H$ and $V$ are permutation matrices, then the horizontal (vertical) neighbor of any symbol is uniquely determined. Further, if $HV=VH$, fixing a symbol at $(0,0)$ determines the symbols at all other positions in the plane and hence yields a unique configuration in the shift space $X_G$. Finally, as $H$ and $V$ are irreducible, any two symbols are connected by a horizontal (vertical) path. Consequently, the shift space comprises of a single periodic orbit (with finite orbit) and the proof is complete.
\end{proof}

\begin{Remark} \label{rem1}
The above proof provides a sufficient condition for the shift space to be finite (in fact a single periodic orbit) when the generating matrices commute with each other.  It may be noted that as the above arguments hold even when only one of the generating matrix is irreducible, the above result holds under a weaker condition (irreducibility of one of the generating matrices). In \cite{markley}, the authors establish transitivity of the shift space under irreducibility of one of the generating matrices when the generating matrices are ``consistent". However, as irreducibility of $H$ (or $V$) ensures connectivity of any two vertices through a horizontal (vertical) path, $(HV)_{ij}\neq 0 \Leftrightarrow (VH)_{ij}\neq 0$ for all $i,j$ ensures that any two allowed rectangular blocks can be connected (horizontally or vertically) and hence can be extended to a configuration in the shift space $X_G$. Consequently, the shift space $X_G$ is transitive for a bigger class of subshifts (of the full shift) and we get the following result.
\end{Remark}

\begin{Proposition}\label{prop2}
Let $X_{G}$ be a shift space generated by the two dimensional graph $G=(\mathcal{H},\mathcal{V})$ such that $(HV)_{ij}\neq 0 \Leftrightarrow (VH)_{ij}\neq 0$ holds for all $i,j$. If $H$ (or $V$) is irreducible then $X_G$ is transitive.
\end{Proposition}

\begin{proof}
Let $H$ be an irreducible matrix and let $x,y\in X_G$. For any central blocks $x_n= {\begin{array}{cccccccc}

 x_{-n,n}  & \hdots &  x_{n,n}  \\
\vdots  & \vdots & \vdots     \\
 x_{-n,-n}  & \hdots & x_{n,-n}      \\
	
	\end{array} } $ and $y_n= {\begin{array}{cccccccc}

 y_{-n,n}  & \hdots &  y_{n,n}  \\
\vdots & \vdots & \vdots     \\
 y_{-n,-n}  & \hdots & y_{n,-n}      \\
	
	\end{array} } $ of $x$ and $y$ respectively, irreducibility of $H$ ensures that $x_{n,n}$ and $y_{-n,-n}$ can be connected through a horizontal path. Consequently, there exists symbols $a_1,a_2,\ldots, a_k$ such that $${\begin{array}{cccccccccccccc}	
&&&&  y_{-n,n} & \hdots  &  y_{n,n}  \\
&&&&  \vdots & \vdots  & \vdots     \\
x_{-n,n} & \hdots  &  x_{n,n} & a_{1}\hdots \ a_{k} & y_{-n,-n}  & \hdots & y_{n,-n} \\
\vdots & \vdots  & \vdots  &&&& \\
x_{-n,-n} & \hdots  & x_{n,-n}  &&&&   \\

			\end{array} }$$\\
is an allowed pattern for the shift space $X_G$. Further, as $(HV)_{ij}\neq 0 \Leftrightarrow (VH)_{ij}\neq 0$ holds for all $i,j$, such an allowed pattern can be extended to a $(4n+k+2)\times (4n+k+2)$ square and hence to a configuration in $X_G$. As the argument holds for any central blocks of any size $n$,~ for any neighborhoods $U$ and $V$ of $x$ and $y$ respectively, there exists $r,s\in \mathbb{N}$ such that $\sigma^{(r,s)}(U)\cap V\neq \emptyset$. As the arguments hold for any $x,y\in X_G$, $X_G$ is transitive and the proof is complete.
\end{proof}

\begin{Remark} \label{rem4}
The above results establish non-emptiness and transitivity under suitable conditions for a two dimensional shift space. While Proposition $1$ establishes the shift space to comprise of a single periodic orbit when generating matrices are irreducible permutation matrices commuting with each other, Proposition $2$ provides sufficient conditions to establish transitivity for the two dimensional shift space respectively. However, the conditions derived are sufficient in nature and the conclusions derived may hold even in the absence of the conditions imposed. We now provide examples in support of our arguments.
\end{Remark}

\begin{ex}
	
	Let $X_{G}$ be a shift space generated by the adjacency matrices given below:
	
	$$\textit{H}= \bordermatrix{ & 1 & 2 & 3 & 4  \cr
		1 & 1 & 1 & 0 & 0   \cr
		2 & 1 & 0 & 0 & 0   \cr
		3 & 0 & 0 & 0 & 1   \cr
		4 & 0 & 0 & 1 & 0   \cr	
	}
	\ \ \ \ \ \ \ \ \
	\textit{V}= \bordermatrix{ & 1 & 2 & 3 & 4  \cr
		1 & 0 & 0 & 1 & 0   \cr
		2 & 0 & 0 & 0 & 1   \cr
		3 & 1 & 1 & 0 & 1   \cr
		4 & 0 & 1 & 0 & 0   \cr	
	}$$
	
	It may be noted that $H$ and $V$ are neither irreducible nor permutation matrices. Further, as $HV\neq VH$, none of the conditions (imposed in Proposition $1$) holds. However, as the shift space comprises of a single periodic orbit of the point
	
	$$ {\begin{array}{ccccccccccccccccccccccc}

			\ldots & \vdots & \vdots & \vdots & \vdots & \vdots  & \vdots & \vdots & \vdots   & \vdots & \vdots & \vdots & \vdots & \vdots & \vdots  & \vdots & \vdots & \vdots   & \ldots   \\
			
			\ldots & 1 & 2 & 1 & 2 & 1  & 2 & 1 & 2  & 1 & 2 & 1 & 2 & 1  & 2 & 1 & 2  & \ldots   \\
			
			\ldots & 3 & 4 & 3 & 4 & 3 & 4 & 3 & 	4 & 3 & 4 & 3 & 4 & 3 &	4 & 3 & 4 & \ldots    \\

			\ldots & 1 & 2 & 1 & 2 & 1  & 2 & 1 & 2  & 1 & 2 & 1 & 2 & 1  & 2 & 1 & 2  & \ldots   \\
			
			\ldots & 3 & 4 & 3 & 4 & 3 & 4 & 3 & 	4 & 3 & 4 & 3 & 4 & 3 &	4 & 3 & 4 & \ldots    \\			
			\ldots & 1 & 2 & 1 & 2 & 1  & 2 & 1 & 2  & 1 & 2 & 1 & 2 & 1  & 2 & 1 & 2  & \ldots   \\			
			\ldots & \vdots & \vdots & \vdots & \vdots & \vdots  & \vdots & \vdots & \vdots   &\vdots & \vdots & \vdots & \vdots & \vdots & \vdots  & \vdots & \vdots & \vdots   &  \ldots   \\	
			
	\end{array} } $$
	$X_G$ comprises of a single periodic point (even when none of the conditions imposed in Proposition $1$ hold).
\end{ex}

\begin{ex}\label{3}
Let $X_{G}$ be the shift space corresponding to graph $G=(\mathcal{H},\mathcal{V})$, where the adjacency matrices for the generating graph are given by

$$\textit{H}= \bordermatrix{ & 0 & 1 & 2 & e & f & g \cr
	0 & 0 & 1 & 0 & 0 & 0 & 0  \cr
	1 & 0 & 0 & 1 & 0 & 0 & 0  \cr
	2 & 1 & 0 & 0 & 1 & 0 & 0  \cr
	e & 0 & 0 & 1 & 0 & 1 & 0  \cr
	f & 0 & 0 & 0 & 0 & 0 & 1  \cr
	g & 0 & 0 & 0 & 1 & 0 & 0  \cr
}
\ \ \ \ \ \ \ \ \
\textit{V}= \bordermatrix{ & 0 & 1 & 2 & e & f & g \cr
	0 & 0 & 1 & 0 & 0 & 0 & 0  \cr
	1 & 0 & 0 & 1 & 0 & 0 & 0  \cr
	2 & 1 & 0 & 0 & 0 & 1 & 0  \cr
	e & 0 & 0 & 0 & 0 & 1 & 0  \cr
	f & 0 & 0 & 1 & 0 & 0 & 1  \cr
	g & 0 & 0 & 0 & 1 & 0 & 0  \cr
}$$
It may be noted that $H$ and $V$ are irreducible but $(HV)_{ij}\neq 0 \Leftrightarrow (VH)_{ij}\neq 0$ for all $i,j$ does not hold. Further, as the pattern $2e$ cannot be extended vertically to form a valid pattern for the shift space, any valid configuration for $X_G$ comprises either of $0,1,2$ or $e,f,g$ and hence the two dimensional shift space $X_G$ is not transitive.  Thus, the shift space may fail to be transitive even when each of the generating matrices are irreducible (It is worth mentioning that if $(HV)_{ij}\neq 0 \Leftrightarrow (VH)_{ij}\neq 0$ holds for all $i,j$, irreducibility of any of the generating matrices is sufficient to ensure transitivity of the two dimensional shift $X_G$).
\end{ex}


\begin{Cor}
	Let $X_{G}$ be a shift space generated by the two dimensional graph $G=(\mathcal{H},\mathcal{V})$ such that $(HV)_{ij}\neq 0 \Leftrightarrow (VH)_{ij}\neq 0$ holds for all $i,j$. If $X_H$ (or $X_V$) is weak mixing then $X_G$ exhibits weak mixing.
\end{Cor}

\begin{proof}
As proof of Proposition \ref{prop2} establishes that if $H$ (or $V$) is irreducible then $(HV)_{ij}\neq 0 \Leftrightarrow (VH)_{ij}\neq 0$ for all $i,j$ ensures transitivity of the shift space $X_G$. As a similar set of arguments can be repeated for any pair of open sets in the shift space, if $X_H$ is weakly mixing then $X_G$ is weak mixing (under the imposed condition).
\end{proof}

\begin{Remark} \label{rem3}
The above result proves that if the imposed conditions hold, then the two dimensional shift is weak mixing when the one dimensional shift generated by any of the generating matrices is weak mixing. However, the result holds strictly under the imposed conditions and may not be true otherwise even when the one dimensional shifts generated by each of the generating matrices are weak mixing. We now give an example in support of our claim.
\end{Remark}

\begin{ex}
Let $X_G$ be the two dimensional shift space generated by graph $G=(\mathcal{H},\mathcal{V})$, where corresponding adjacency matrices are

	$$\textit{H}= \bordermatrix{ & 0 & 1 & 2 \cr
		0 & 1 & 1 & 0 \cr
		1 & 0 & 0 & 1 \cr
		2 & 1 & 0 & 0 \cr 	
	}
	\ \ \ \ \ \ \ \ \
	\textit{V}= \bordermatrix{ & 0 & 1 & 2 \cr
		0 & 0 & 1 & 1 \cr
		1 & 1 & 0 & 0 \cr
		2 & 1 & 0 & 1 \cr 	
	}$$
Firstly, it may be noted that $H$ and $V$ are primitive matrices (hence generate one dimensional weak mixing shifts) and the condition $(HV)_{ij}\neq 0 \Leftrightarrow (VH)_{ij}\neq 0$ for all $i,j$ fails to hold. However, the two dimensional shift generated by $G$ comprises of a single periodic orbit generated by the point
	
	$$ {\begin{array}{ccccccccccccccccccccccc}
			
			\ldots & \vdots & \vdots & \vdots & \vdots & \vdots  & \vdots & \vdots & \vdots   & \vdots & \vdots & \vdots & \vdots & \vdots & \vdots  & \vdots & \vdots & \vdots   & \ldots   \\

			\ldots & 1 & 2 & 0 & 0 & 1  & 2 & 0 & 0  & 1 & 2 & 0 & 0 & 1  & 2 & 0 & 0  & \ldots   \\
			
			\ldots &	0 & 0 & 1 & 2 & 0 & 0 &	1 & 2 &	0 & 0 & 1 & 2 & 0 & 0 &	1 & 2 & \ldots    \\
			
%
			
			\ldots & 1 & 2 & 0 & 0 & 1  & 2 & 0 & 0  & 1 & 2 & 0 & 0 & 1  & 2 & 0 & 0  & \ldots   \\
			\ldots & \vdots & \vdots & \vdots & \vdots & \vdots  & \vdots & \vdots & \vdots   &\vdots & \vdots & \vdots & \vdots & \vdots & \vdots  & \vdots & \vdots & \vdots   &  \ldots   \\	
	\end{array} } $$\\
Consequently, the shift space $X_G$ does not exhibit weak mixing. Thus, the shift space $X_G$ generated by a graph $G$ need not exhibit weak mixing even when the shift spaces generated by each of the generating matrices are weak mixing.
\end{ex}

\begin{Proposition}\label{ttcc}
	Let $G=(H,V)$ be a two dimensional connected graph satisfying $(HV)_{ij} \neq 0 \Leftrightarrow (VH)_{ij} \neq 0$ for all $i, j$, then $X_G$ is transitive.
\end{Proposition}

\begin{proof}
Let $U, V$ be $\frac{1}{n}$-neighborhoods of configurations $x, y \in X_G$. If $x_n$ and $y_n$ are central blocks of length $(2n+1)$ of the elements $x$ and $y$ respectively, as $G$ is connected, there exists a path from $x_{n,n}$ (top right corner of $x_n$) to $y_{-n,-n}$ (bottom left corner of $y_n$). As $(HV)_{ij} \neq 0 \Leftrightarrow (VH)_{ij} \neq 0$ (for all $i, j$), the extended pattern can further be extended to an allowed rectangular pattern in $\mathcal{B}(X_G)$ (and finally to a valid configuration $z \in X_G$). Thus, there exists $(r,s)\in (\mathbb{Z}^{+} \times \mathbb{Z}^{+})\setminus\{0,0\}$ such that $\sigma^{(r,s)}(U) \cap V \neq \emptyset$. Hence $X_G$ is transitive and the proof is complete.
\end{proof}

\begin{Proposition} \label{5}
Let $X_{G}$ be a shift space and let $(HV)_{ij}\neq 0 \Leftrightarrow (VH)_{ij}\neq 0$ holds for all $i,j$. Then, $X_{G}$ is doubly $(r,s)$-transitive ($rs>0)$ if and only if $H^{r}V^{s}$ is irreducible.
\end{Proposition}

\begin{proof}
 Without loss of generality, assume that $r,s>0$ (as $rs>0)$. Let  $X_{G}$ be doubly $(r,s)$-transitive ($rs>0)$, $x,y\in X_G$ and $U, V$ be $\frac{1}{n}$-neighborhoods of $x$ and $y$ respectively. Let $x_n, y_n$ be central blocks (of size $2n+1$) of the elements $x$ and $y$ respectively. As $X_G$ is transitive in the direction $(r,s)$, there exists $u\in U$, $v\in V$ and $k\in\mathbb{N}$ such that $\sigma^{(kr,ks)}(u)=v$. As $U, V$ are $\frac{1}{n}$-neighborhoods of $x, y$ respectively, $x_n, y_n$ are central blocks of $u, v$ respectively and hence $(H^{r}V^{s})^{k}_{x_{n,n}y_{-n,-n}}\neq 0$. As the proof holds with central blocks of arbitrary sizes (and hence with arbitrary corner points), $H^{r}V^{s}$ is irreducible and the proof of forward part is complete. \\

Conversely, let $H^{r}V^{s}$ be irreducible ($r, s>0$), $x,y\in X_G$, $U$ and $V$ be $\frac{1}{n}$-neighborhoods of $x$ and $y$ respectively and $x_n, y_n$ be central blocks (of size $2n+1$) of the elements $x, y$ respectively. For any central blocks $x_n$ and $y_n$ (of size $2n+1$) of $x$ and $y$ respectively, irreducibility of $H^{r}V^{s}$ ensures existence of $k\in\mathbb{N}$ such that $(H^{r}V^{s})^{k}_{x_{n,n}y_{-n,-n}}\neq 0$. Consequently, there exists a pattern of the form $${\begin{array}{cccccccccccccc}	
		&& &&&&& y_{-n,n} & \hdots  &  y_{n,n}  \\
		&& &&&&& \vdots & \vdots  & \vdots     \\
		&& &&&&& y_{-n,-n}  & \hdots & y_{n,-n} \\
		
			&&&&&&& \vdots & \\
		&&&&&& \hdots  \\
		&&&&& \reflectbox{$\ddots$}  \\
		&&&& \vdots &&& \\
		x_{-n,n} & \hdots  &  x_{n,n} &  \hdots  &&&& \\
		\vdots & \vdots  & \vdots  &&&&& \\
		x_{-n,n} & \hdots  & x_{n,-n}  &&&&&   \\
\end{array} }$$\\
(where each $(\hdots)$ is a horizontal path of length $r$ and each $(\vdots)$ is a vertical path of length $s$) which is an allowed pattern for the shift space $X_G$. Further, as $(HV)_{ij}\neq 0 \Leftrightarrow (VH)_{ij}\neq 0$ holds for all $i,j$, such an allowed pattern can be extended to a configuration (say $z$) in $X_G$. Finally, note that $z\in U$ (as central block of $z$ is $x_n$) and $\sigma^{(kr,ks)}(z)\in V$ (as central block of $\sigma^{(kr,ks)}(z)$ is $y_n$), $X_G$ is doubly $(r,s)$-transitive and the converse also holds.

Finally, if $r,s<0$ (as $rs>0$) then arguments similar to above prove that $X_G$ is $(r,s)$-transitive  if and only if $(H^T)^{r}(V^T)^{s}$ is irreducible. Since $A$ is irreducible if and only if $A^T$ is irreducible,  $X_{G}$ is doubly $(r,s)$-transitive ($rs>0)$ if and only if $H^{r}V^{s}$ is irreducible and the proof is complete.
\end{proof}

\begin{Remark} \label{rem5}
The above result establishes the $(r,s)$-doubly transitivity of the two dimensional shift space under the irreducibility of the matrix $H^rV^s$. However, as the interactions among symbols may happen in negative time, irreducibility of $H^rV^s$ cannot be guaranteed under $(r,s)$-transitivity of the shift space under discussion. It may be noted that a similar set of arguments establish weak mixing in the $(r,s)$-direction when $H^rV^s$ is primitive (under the condition $(HV)_{ij}\neq 0 \Leftrightarrow (VH)_{ij}\neq 0$ for all $i,j$). We now establish our claims below. \\ \\

\begin{ex}
Let $X_G$ be a shift space arising from the following adjacency matrices:
	$$\textit{H}= \bordermatrix{ & 1 & 2 & 3 & 4  \cr
		1 & 0 & 1 & 0 & 0 \cr
		2 & 1 & 0 & 1  & 0 \cr
		3 & 0 & 0 & 0  & 1 \cr   
		4 & 0 & 0 & 1  & 0 \cr  
	}
	\ \ \ \ \ \
	\textit{V}= \bordermatrix{ & 1 & 2 & 3  & 4 \cr
		1 & 1 & 0 & 0 & 0   \cr
		2 & 0 & 1 & 0 & 0   \cr
		3 & 0 & 0 & 1 & 0   \cr
		4 & 0 & 0 & 0 & 1   \cr	
	}
	$$
Then, it can be seen that $X_{G}$ is (1,1)-transitive (is transitive along the line $y=x$). However, as
	$$\textit{HV=VH}=  \bordermatrix{ & 1 & 2 & 3 & 4  \cr
		1 & 0 & 1 & 0 & 0 \cr
		2 & 1 & 0 & 1  & 0 \cr
		3 & 0 & 0 & 0  & 1 \cr   
		4 & 0 & 0 & 1  & 0 \cr  
	}
	$$
The matrix $HV$ fails to be irreducible and hence $(r,s)$-transitivity of the shift space $X_G$ need not imply irreducibility of the matrix $H^{r}V^{s}$.
\end{ex}
\end{Remark}

\begin{Cor}
	Let $X_{G}$ be a shift space where condition $(HV)_{ij}\neq 0 \Leftrightarrow (VH)_{ij}\neq 0$ holds for all $i,j$. Then, $H^{r}{V}^{s}$ is primitive $\Leftrightarrow$ $X_{G}$ is weak mixing in direction $(r,s)$ ($ rs >0)$.
\end{Cor}

\begin{proof}
	The proof follows from the discussions in the Remark \ref{rem5}.
\end{proof}

\begin{Proposition} \label{tr}
Let $X_G$ be a one dimensional shift space generated by a one dimensional graph $G$ (with adjacency matrix $A$). Then $X_G$ is transitive $\Leftrightarrow$ $A$ is irreducible or $A$ has exactly two irreducible, permutation sub-matrices $A_{1}, A_{2}$ such that for any $i \in A_{1}$, $j \in A_{2}$, there exists a unique path connecting the two vertices (which traverses through all the vertices of $A \setminus (A_1 \cup A_2)$).
\end{Proposition}

\begin{proof}
	
Let $X_G$ be a one dimensional transitive shift space generated by a one dimensional graph $G$ (with adjacency matrix $A$). Note that if $A$ is not irreducible then there exists $i,j\in\{1,2,\ldots,n\}$ such that $A^k_{ij}=0~~ \forall k\in\mathbb{N}$. However,  as $X_G$ is transitive, there exists a block of the form $ja_1a_2\ldots a_r i$ allowed for the shift space. Further, as existence of two blocks of the form $ja_1a_2\ldots a_r i$ ensures $A^k_{ij}\neq 0$ for some $k\in\mathbb{N}$ (as $X_G$ is transitive), there exists a unique block of the form $ja_1a_2\ldots a_r i$. Thus, fixing position of $j$ fixes every position before $j$ in a unique manner. Also, as existence of two blocks of the form $ja_1a_2\ldots a_r i b_1b_2\ldots b_m$ ensures $A^k_{ij}\neq 0$ for some $k\in\mathbb{N}$ (again by transitivity of $X_G$), fixing position of $j$ fixes the entire sequence and the proof of forward part is complete.\\
	
Conversely, if hypothesis for the converse holds then any non-periodic point in the shift space is of the form $\overline{x_1x_2\ldots x_k} j c_1 c_2\ldots c_l i \overline{y_1y_2\ldots y_m}$  Also, as $\overline{x_1x_2\ldots x_k}$ and $\overline{y_1y_2\ldots y_m}$ are the only other elements in the shift space, any two blocks allowed for the shift space interact over time and the shift space is transitive.   	
\end{proof}

\begin{Remark}\label{rem6}
The above result characterizes the transitivity for a one dimensional shift space generated by a graph $G$. However, as directional shifts for a multidimensional shift spaces can be visualized as a one dimensional shift, a similar result can be derived to characterize the transitivity (in a particular direction) for a multidimensional shift space. In particular, for a two dimensional shift space generated by $G=(\mathcal{H},\mathcal{V})$, if $(HV)_{ij}\neq 0 \Leftrightarrow (VH)_{ij}\neq 0$ holds (for all $i,j$) then a similar argument establishes $(r,s)$-transitivity for the shift space if $H^rV^s$ can be represented using the similar block form (as derived in Proposition \ref{tr}). Thus we get the following corollary.
\end{Remark}

\begin{Cor}
Let $X_{G}$ be a two dimensional shift space generated by the graph $G=(\mathcal{H},\mathcal{V})$ such that $(HV)_{ij}\neq 0 \Leftrightarrow (VH)_{ij}\neq 0$ holds for all $i,j$. Then $X_{G}$ is $(r,s)$-transitive $\Leftrightarrow$ $H^rV^s$ is irreducible or $H^rV^s$ has exactly two irreducible, permutation sub-matrices $A_{1}, A_{2}$ such that for any $i \in A_{1}$, $j \in A_{2}$, there exists a unique path connecting the two vertices (which traverses through all the vertices of $H^rV^s \setminus (A_1 \cup A_2)$).
\end{Cor}

\begin{proof}
The proof follows from discussions in Remark \ref{rem6}.
\end{proof}



%
%

\begin{Proposition}
Let $X_{G}$ be a two dimensional shift space and let $(HV)_{ij}\neq 0 \Leftrightarrow (VH)_{ij}\neq 0$ holds for all $i,j$. If $H$ or $V$ are irreducible matrices, then $X_{G}$ is doubly $(r,s)$-transitive (for some direction  $(r,s)\in\mathbb{N}\times\mathbb{N} $).
\end{Proposition}
\begin{proof}
Let $X_G$ be the shift generated by the two dimensional graph $G$ such that $(HV)_{ij}\neq 0 \Leftrightarrow (VH)_{ij}\neq 0$ holds for all $i,j$ and let $H$ be an irreducible matrix. As $V$ generates a non-empty one dimensional shift, there exists a vertex $i\in\{1,2,\ldots,n\}$ and $m\in\mathbb{N}$ such that $V^m_{ii}\neq 0$. Also, as $H$ is irreducible, there exists a finite path $W=a_0a_1\ldots a_r$ such that $a_0=a_r=i$ and each symbol $j$ appears atleast once in the path $W$. Consequently, for the pattern $P$ of the form

$${\begin{array}{cccccccccccccc}	
		&&&&& i & a_{1} & a_{2} & \ldots & a_{r-1} &  i\\
		&  &  &      & & \vdots &&&&& \\
		
		i & a_{1} & a_{2} & \ldots & a_{r-1} &  i &&
	\end{array} } $$\\		
is allowed for the shift space $X_G$. Further, note that diagonal repetition ($R$ times, where $R$ is more than the number of symbols) is allowed for the shift space and can be extended to a configuration in $X_G$. Also, as $\sigma^{(s(r+1),sm)}(i)=a_s$, $H^{r+1}V^m$ is irreducible. Consequently, $X_G$ is doubly $(r+1,m)$-transitive and the proof is complete.		
\end{proof}

\begin{Remark} \label{rem8}
The above proof establishes that if $(HV)_{ij}\neq 0 \Leftrightarrow (VH)_{ij}\neq 0$ holds for all $i,j$, then irreducibility of any of the generating matrices ensures directional transitivity (in some direction $(r,s)$) and thus transitivity cannot happen in the shift space independently (without directional transitivity). However, in general, the transitivity of the shift space $X_G$ need not imply directional transitivity and the notions are indeed distinct. Also, as weakly mixing ensures interaction between any two pairs of symbols, if any of the shift spaces generated by $H$ or $V$ is weakly mixing, $(HV)_{ij}\neq 0 \Leftrightarrow (VH)_{ij}\neq 0$ for all $i,j$ ensures weak mixing in some direction $(r,s)$ for the two dimensional shift space. It may be noted that as a shift is $(r,s)$-weak mixing if and only if $H^rV^s$ is primitive (for $rs>0$), $(r,s)$-weak mixing of a system is equivalent to $(1,1)$-weak mixing.  It is worth mentioning that directional mixing of a system indeed depends on the direction and existence of mixing notions in one direction for a shift space $X_G$ need not guarantee existence of mixing notion in the other. Consequently, total transitivity and weak mixing are indeed distinct notions for higher dimensional shifts of finite type (recall that the two notions coincide for shifts of finite type in one dimensional case).  We now establish our claims below.
\end{Remark}

\begin{Cor}
	Let $X_{G}$ be a shift space such that $(HV)_{ij}\neq 0 \Leftrightarrow (VH)_{ij}\neq 0$ for all $i,j$. If $H$ or $V$ is primitive, then $X_{G}$ is $(r,s)$-weak mixing (for some direction  $(r,s)\in\mathbb{N}\times\mathbb{N}$). Further, $X_{G_{}} $ is $(r,s)$-weak mixing (for $rs > 0$) if and only if $X_{G} $ is $(1,1)$-weak mixing.
\end{Cor}

\begin{proof}
The proof follows from the discussions in Remark \ref{rem8}.
\end{proof}

\begin{ex}
	Let $X_G$ be a shift space arising from following adjacency matrices.
	
	$$\textit{H = V}= \bordermatrix{ & 0 & 1 & 2 & 3  \cr
		0 & 0 & 1 & 0 & 0   \cr
		1 & 0 & 1 & 1 & 0   \cr
		2 & 1 & 0 & 0 & 1   \cr
		3 & 1 & 0 & 0 & 1   \cr	
	}
	$$
	It may be noted that as $H$ is primitive ($H^4_{ij}>0 \  \forall \  i,j$), $X_G$ is $(1,1)$-weak mixing. However, $HV^T=V^TH$ fails to hold. Further, as
	
	$$\textit{$HV^{T}$}= \bordermatrix{ & 0 & 1 & 2 & 3  \cr
		0 & 1 & 1 & 0 & 0   \cr
		1 & 1 & 2 & 0 & 0   \cr
		2 & 0 & 0 & 2 & 2   \cr
		3 & 0 & 0 & 2 & 2   \cr	
	}
	\ \ \ \ \ \
	\textit{$V^{T}H$}= \bordermatrix{ & 0 & 1 & 2 & 3  \cr
		0 & 2 & 0 & 0 & 2   \cr
		1 & 0 & 2 & 1 & 0   \cr
		2 & 0 & 1 & 1 & 0   \cr
		3 & 2 & 0 & 0 & 2   \cr	
	}
	$$\\
fixing any symbol at $(0,0)$ enforces the same symbol along the line $y=-x$ and hence $X_G$ does not exhibit (1,-1)-weak mixing. Consequently, any of the mixing notions in one direction need not ensure the mixing notions in the other. Also, as the example exhibits weak mixing in $(1,1)$-direction, the case provides an example of weak mixing shift which fails to be totally transitive. Consequently, the two notions are distinct for higher dimensional shifts of finite type and an analogous extension of the result valid for one dimensional case does not hold.
\end{ex}


Let $P_k$ and $Q_k$ denote the collection of valid $1\times k$ and $k\times 1$ patterns respectively and let $V_k (H_k)$ denote the matrix characterizing vertical (horizontal) compatibility of elements of $P_k$ ($Q_k$). It may be noted that $H_1=H$ and $V_1=V$. In general, $H_k$ and $V_k$ generate valid patterns of size $k\times r$ and $r\times k$ respectively (for any $r\in \mathbb{N}$). Such matrices have been defined in \cite{markley1} and have been used to investigate the dynamical properties of ``shifts of finite type"(as defined in \cite{markley,markley1}). We now investigate notions of horizontal (vertical) mixing for the shift space $X_G$ in terms of the matrices $H_k$ and $V_k$.

\begin{Proposition}
	Let $X_G$ be a two dimensional shift space generated by a graph $G$ and let $(HV)_{ij}\neq 0 \Leftrightarrow (VH)_{ij}\neq 0$ holds for all $i,j$. Then, $X_G$ is horizontally doubly transitive $\Leftrightarrow~~$ $H_k$ is irreducible $\forall k\in \mathbb{N}$.
\end{Proposition}

\begin{proof}
	Let $X_G$ be a two dimensional horizontally doubly transitive shift space. Let $x,y\in X_G$, $x_k= {\begin{array}{cccccccc}
			
			x_{-k,k}  & \hdots &  x_{k,k}  \\
			\vdots  & \vdots & \vdots     \\
			x_{-k,-k}  & \hdots & x_{k,-k}      \\
			
	\end{array} } $ and $y_k= {\begin{array}{cccccccc}
			
			y_{-k,k}  & \hdots &  y_{k,k}  \\
			\vdots & \vdots & \vdots     \\
			y_{-k,-k}  & \hdots & y_{k,-k}      \\
			
	\end{array}}~$, be central blocks and let $U, V$ be $\frac{1}{k}$-neighborhoods of $x, y$ respectively. As $X_G$ is horizontally doubly transitive, there exists $z\in U$, $w\in V$ and $r\in \mathbb{N}$ such that $\sigma^{(r,0)}(z)=w$. It may be noted that as $z\in U$ and $w\in V$, the central blocks of $z$ and $w$ are $x_k$ and $y_k$ respectively. Further, as $\sigma^{(r,0)}(z)=w$,  $a={\begin{array}{cccccccc} x_{k,k}\\ \vdots \\ x_{k,-k}\\ \end{array}}$ and $b={\begin{array}{cccccccc} y_{-k,k}\\ \vdots \\ y_{-k,-k}\\ \end{array}}$ can be connected horizontally (through a finite path). As the arguments hold for any $a,b\in H_k$, $H_k$ is irreducible. Further, as the argument holds for any $k\in \mathbb{N}$, $H_k$ is irreducible for all $k\in\mathbb{N}$ and the proof of forward part is complete.\\
	
	Conversely, let $H_k$ be irreducible for all $k\in\mathbb{N}$. Let $x,y\in X_{G}$ and let $U$ and $V$ be $\frac{1}{k}$-neighborhoods of $x$ and $y$ respectively. For any central blocks $x_k= {\begin{array}{cccccccc}
			
			x_{-k,k}  & \hdots &  x_{k,k}  \\
			\vdots  & \vdots & \vdots     \\
			x_{-k,-k}  & \hdots & x_{k,-k}      \\
			
	\end{array} } $ and $y_k= {\begin{array}{cccccccc}
			
			y_{-k,k}  & \hdots &  y_{k,k}  \\
			\vdots & \vdots & \vdots     \\
			y_{-k,-k}  & \hdots & y_{k,-k}      \\
	\end{array}}~$ of $x$ and $y$ respectively, irreducibility of $H_k$ ensures that $a={\begin{array}{cccccccc} x_{k,k}\\ \vdots \\ x_{k,-k}\\ \end{array}}$ and $b={\begin{array}{cccccccc} y_{-k,k}\\ \vdots \\ y_{-k,-k}\\ \end{array}}$ can be connected horizontally through a finite path (say of length $l$) and hence there exists a valid pattern of the form $${\begin{array}{cccccccc}
			
			x_{-k,k}  & \hdots &  x_{k,k} & \hdots & y_{-k,k}  & \hdots &  y_{k,k} \\
			\vdots  & \vdots & \vdots  & \vdots & \vdots & \vdots & \vdots    \\
			x_{-k,-k}  & \hdots & x_{k,-k}  & \hdots &  y_{-k,-k}  & \hdots & y_{k,-k}   \\		
	\end{array} } $$\\
	As $(HV)_{ij}\neq 0 \Leftrightarrow (VH)_{ij}\neq 0$ holds for all $i,j$, the above pattern can be extended to a valid configuration in $X_G$ (say $z$). Then, $z\in U$, $\sigma^{(2k+1+l,0)}(z)\in V$. As the proof holds for any $x,y\in X_G$, the shift space is horizontally doubly transitive and the proof is complete.
\end{proof}

\begin{Remark}\label{rem9}
The above result establishes that for any shift space $X_G$, $X_G$ is horizontally doubly transitive if and only if each $H_k$ is irreducible. However, transitivity of $X_G$ need not ensure irreducibility for all $H_k$. Once again, as horizontal transitivity can be visualized as a one dimensional shift, horizontal transitivity of the shift space can be characterized via the block representations of the matrices $H_k$. Finally, as similar arguments establish the equivalence of vertical transitivity with irreducibility of $V_k$ (for all $k$), we get the following corollary.
\end{Remark}

\begin{Cor}
	Let $X_{G}$ be two dimensional shift space generated by the graph $G=(\mathcal{H},\mathcal{V})$ such that $(HV)_{ij}\neq 0 \Leftrightarrow (VH)_{ij}\neq 0$ for all $i,j$. Then, $X_G$ is horizontally transitive $\Leftrightarrow$ either all $H_k$ are irreducible or all $H_{k}$ have exactly two irreducible, permutation sub-matrices $A_{1k}, A_{2k}$ such that for any $i \in A_{1k}$, $j \in A_{2k}$, there exists a unique path connecting the two vertices (which traverses through all the vertices of $H_k \setminus (A_{1k} \cup A_{2k})$).
\end{Cor}

\begin{proof}
	The proof follows from discussions in Remark \ref{rem9}.
\end{proof}

%
%

\begin{Proposition}
	Let $X_G$ be a two dimensional shift generated by the graph $G$ and let $(HV)_{ij}\neq 0 \Leftrightarrow (VH)_{ij}\neq 0$ holds for all $i,j$. If $X_H$ is weakly mixing then $X_G$ is vertically doubly transitive $\Leftrightarrow$ $X_G$ is vertically weakly mixing.
\end{Proposition}

\begin{proof}
	
Let $X_G$ be vertically doubly transitive shift space and let $(HV)_{ij}\neq 0 \Leftrightarrow (VH)_{ij}\neq 0$ holds for all $i,j$. Let $x,y,z,w\in X_G$ and \\ $x_k= {\begin{array}{cccccccc}
			
			x_{-k,k}  & \hdots &  x_{k,k}  \\
			\vdots  & \vdots & \vdots     \\
			x_{-k,-k}  & \hdots & x_{k,-k}      \\
			
	\end{array} } $, $y_k= {\begin{array}{cccccccc}
			
			y_{-k,k}  & \hdots &  y_{k,k}  \\
			\vdots & \vdots & \vdots     \\
			y_{-k,-k}  & \hdots & y_{k,-k}      \\
	\end{array}}~$, $z_k= {\begin{array}{cccccccc}
			
			z_{-k,k}  & \hdots &  z_{k,k}  \\
			\vdots  & \vdots & \vdots     \\
			z_{-k,-k}  & \hdots & z_{k,-k}      \\
			
	\end{array} } $ and $w_k= {\begin{array}{cccccccc}
			
			w_{-k,k}  & \hdots &  w_{k,k}  \\
			\vdots & \vdots & \vdots     \\
			w_{-k,-k}  & \hdots & w_{k,-k}      \\
			
	\end{array} } $ be central blocks of size $k$ (squares of size $2k+1$) of $x,y,z$ and $w$ respectively. Let $U,V,Z,W$ be $\frac{1}{k}$-neighborhoods of $x,y,z$ and $w$ respectively. As $X_H$ is weakly mixing, the symbols $x_{k,k}$ and $z_{k,-k}$ can be connected to $y_{-k,k}$ and $w_{-k,-k}$ respectively through  horizontal paths of same length (say $l$). Also, as $V_k$'s are irreducible for all $k$, the strips $x_{-k,k}\ldots x_{k,k}\ldots y_{-k,k}\ldots y_{k,k}$ and $z_{-k,-k}\ldots z_{k,-k}\ldots w_{-k,-k}\ldots w_{k,-k}$ can be connected vertically (say through a path of length $p$).  Consequently, there exists patterns of the form
	
	$${\begin{array}{cccccccccccccc}	
			z_{-k,k} & \hdots  &  z_{k,k} &\hskip 1.5cm&  w_{-k,k} & \hdots  & w_{k,k} \\
			\vdots & \vdots  & \vdots   && \vdots & \vdots  & \vdots  \\
			z_{-k,-k}  & \hdots & z_{k,-k}  && w_{-k,-k}  & \hdots & w_{k,-k}\\
			&&&&&&&&&& \\
			\vdots&\vdots& \vdots&&\vdots&\vdots&\vdots&&&& \\
			\vdots&\vdots& \vdots&&\vdots&\vdots&\vdots&&&& \\
			&&&&&&&&&& \\
			x_{-k,k} & \hdots  &  x_{k,k} &&  y_{-k,k} & \hdots  & y_{k,k} \\
			\vdots & \vdots  & \vdots   &&  \vdots & \vdots  & \vdots  \\
			x_{-k,-k}  & \hdots & x_{k,-k}  &&   y_{-k,-k}  & \hdots & y_{k,-k}\\
	\end{array} } $$\\
	
	which are valid for the shift space $X_G$ and hence can be extended to  configurations in $X_G$ (say $r$ and $s$ respectively). Finally, note that while $r\in U, s\in V$ (as central blocks of $r$ and $s$ are $x_k$ and $y_k$ respectively), we also have $\sigma^{(0,2k+1+p)}(r)\in Z$ and  $\sigma^{(0,2k+1+p)}(s)\in W$ (as their central blocks are $z_k$ and $w_k$ respectively). Consequently, $X_G$ exhibits weak mixing and the proof of forward part is complete.
	
	Conversely, if $X_G$ exhibits vertical weak mixing then it exhibits vertical doubly transitivity and the proof is complete.
\end{proof}

\begin{Remark} \label{rem10}
The above results establish the equivalence of vertical doubly transitivity and vertical weak mixing under the condition $(HV)_{ij}\neq 0 \Leftrightarrow (VH)_{ij}\neq 0$ for all $i,j$ when $H$ exhibits weak mixing. For stronger notions of mixing in any direction $(r,s)$, it may be noted that as shift in a particular direction is a one dimensional map and stronger notions of mixing (such as total transitivity, weak mixing and strong mixing) are equivalent for one dimensional shifts, similar arguments establish analogous results for various notions of directional mixing.  It may be noted that while any of the mixing notions for one of the generating matrices ensures the same for $X_G$ under $(HV)_{ij}\neq 0 \Leftrightarrow (VH)_{ij}\neq 0$ for all $i,j$, any of the mixing notions for $X_G$ need not ensure the same for any of the generating matrices under the conditions imposed. We now give an example in support of our claim.
\end{Remark}

\begin{ex}
	Let $X_G$ be a shift space arising from following adjacency matrices:
	$$\textit{H}= \bordermatrix{ & 0 & 1 & 2 & 3  \cr
		0 & 1 & 1 & 1 & 1   \cr
		1 & 1 & 1 & 1 & 1   \cr
		2 & 0 & 0 & 1 & 1  \cr
		3 & 0 & 0 & 1 & 1   \cr	
	}
	\ \ \ \ \ \
	\textit{V}= \bordermatrix{ & 0 & 1 & 2 & 3  \cr
		0 & 1 & 1 & 0 & 0   \cr
		1 & 1 & 1 & 0 & 0   \cr
		2 & 1 & 1 & 1 & 1   \cr
		3 & 1 & 1 & 1 & 1   \cr	
	}
	$$
	
It may be noted that $H$ and $V$ are not irreducible. However, as $(HV)_{ij} \neq 0 \Leftrightarrow (VH)_{ij} \neq 0 \ \forall \ i,j$ holds and $HV$ is primitive, the shift space $X_G$ exhibits $(1,1)$-weak mixing. Consequently, weak mixing (directional weak mixing) for $X_G$ need not ensure weak mixing (any form of mixing) for any of the generating matrices.
\end{ex}	
 We now turn our attention to the graph shifts arising from graph products between two-dimensional graphs $G_i= (H_i, V_i)$. It may be noted that if a two-dimensional graph $G$ can be expressed as a graph product of some two-dimensional graphs, then dynamics of shift space $X_G$ can be investigated through graphs of reduced complexity and hence provides a more efficient way of investigating the dynamics of the shift space under discussion. Before we move further, we first provide an example in support of our claim.
 
 \begin{ex}
 	Let $G$ be a two-dimensional graph generated by matrices 
 	
 	$$\textit{H}= \bordermatrix{ & a & b & c & d & e & f  \cr
 		a & 1 & 1 & 1 & 1 & 0 & 0 \cr
 		b & 1 & 0 & 1  & 0 & 0 & 0 \cr
 		c & 0 & 0 & 0  & 0 & 1 & 1 \cr   
 		d & 0 & 0 & 0  & 0  & 1 & 0\cr  
 		e & 1 & 1 & 0  & 0  & 0 & 0\cr 
 		f & 1 & 0 & 0  & 0  & 0 & 0\cr 
 	}
 	\ \ \ \ \ \
 	\textit{V}= \bordermatrix{ & a & b & c & d & e & f  \cr
 		a & 0 & 0 & 1 & 1 & 1 & 1 \cr
 		b & 0& 0 & 1  & 0 & 1 & 0 \cr
 		c & 1 & 1 & 0  & 0 & 0 & 0 \cr   
 		d & 1 & 0 & 0  & 0  & 0 & 0\cr  
 		e & 1 & 1 & 0  & 0  & 0 & 0\cr 
 		f & 1 & 0 & 0  & 0  & 0 & 0\cr 
 	}
 	$$
 	Next, consider matrices 
 	$$\textit{$H_1$}= \bordermatrix{ & 0 & 1 & 2  \cr
 		0 & 1 & 1 & 0   \cr
 		1 & 0 & 0 &  1  \cr	
 		2 & 1 & 0 &  0  \cr
 	},
 	\ \ \ \ 
 	\textit{$ V_1 $}= \bordermatrix{ & 0 & 1 & 2  \cr
 		0 & 0 & 1 & 1   \cr
 		1 & 1 & 0 & 0   \cr
 		2 & 1 & 0 & 0   \cr
 	}
 	$$ and
 	$$\textit{$H_2$}= \bordermatrix{ & 3 & 4  \cr
 		3 & 1 & 1   \cr
 		4 & 1 & 0  \cr	
 	},
 	\ \ \ \ 
 	\textit{$ V_2 $}= \bordermatrix{ & 3 & 4  \cr
 		3 & 1 & 1   \cr
 		4 & 1 & 0  \cr	
 	}
 	$$
 	It can be observed that $ \exists$ indices $i,j$ such that $(HV)_{ij}=0$ but $(VH)_{ij}\neq 0$ and $ \exists$ indices $k,l$ such that $(VH)_{kl}=0$ but $(HV)_{kl}\neq 0$. However, as $H=H_1 \times H_2$ and $V=V_1 \times V_2$, the two-dimensional graph $G=(H,V)$ can be expressed as the Tensor product $G_1 \times G_2$, where $G_i= (H_i, V_i)$ for $i \in \{1,2\}$. Consequently, $X_G$ is conjugate to $X_{G_1 \times G_2}$ and thus the qualitative nature of the shift space can be investigated through dynamics of ``much simpler spaces" $X_{G_1}$ and $X_{G_2}$. In particular, it can be seen that $X_{G_1}$ and $X_{G_2}$ are non-empty, have horizontal (vertical) periodic points and are $(1,1)$-transitive. Consequently, the shift space $X_G$ is indeed non-empty, has horizontal (vertical) periodic points and is also $(1,1)$-transitive. Consequently, the dynamics of the shift space $X_G$ can be investigated through dynamics of ``much simpler spaces" $X_{G_1}$ and $X_{G_2}$.
 \end{ex}
 \begin{Proposition}\label{1} 
 	For any given graph $G=(V,E)$ with $n$ vertices, $G$ can be expressed as a Tensor product of two graphs if and only if there exists $r_1,r_2\in\mathbb{N}$ such that $n=r_1r_2$, a $r_2\times r_2$ matrix $B$ and a permutation $\{v_{\sigma(1)},v_{\sigma(2)}, \ldots, v_{\sigma(n)} \}$ of the vertex set $V$ such that adjacency matrix $A_G$ (with respect to $\{v_{\sigma(1)},v_{\sigma(2)},\ldots,v_{\sigma(n)} \}$) is of the form
 	
 	$\bordermatrix{ & &  &  & \cr
 		&	S_{11} & S_{12} & \ldots & S_{1 r_2} \cr
 		&   S_{21} & S_{22} & \ldots & S_{2 r_2} \cr
 		&   \vdots & \vdots & \vdots & \vdots \cr   
 		&   S_{r_11} & S_{r_12} & \ldots & S_{r_1 r_2} \cr
 	}
 	\ \ \ \ \ \ $  where $S_{ij}$ are either $B$ or $0$ (the zero matrix)
 \end{Proposition}
 
 \begin{proof}
 	First, consider that graph $G$ can be expressed as a Tensor product of two graphs $G_1, G_2$, where $\mathcal{V}(G_1)=\{a_1, a_2 \ldots a_{r_1}\}$ and $\mathcal{V}(G_2)=\{b_1, b_2 \ldots b_{r_2}\}$. Then, if the vertex set for Tensor product is ordered as $(a_1, b_1), (a_1, b_2)  \ldots (a_1, b_{r_2}), (a_2, b_1), (a_2, b_2)\ldots (a_2,$ $   b_{r_2}) \ldots (a_{r_1}, b_1), (a_{r_1}, b_2)\ldots $ $(a_{r_1}, b_{r_2})$, as $(a_i, b_j) (a_k, b_l) \in \mathcal{E}(G_1 \times G_2)$  if and only if  $a_i a_k \in \mathcal{E}(G_1)$ and  $b_j b_l \in \mathcal{E}(G_2)$, the adjacency matrix for the Tensor product is given by 
 	
 	$\bordermatrix{ & &  &  & \cr
 		&	S_{11} & S_{12} & \ldots & S_{1r_2} \cr
 		&   S_{21} & S_{22} & \ldots & S_{2r_2} \cr
 		&   \vdots & \vdots & \vdots & \vdots \cr   
 		&   S_{r_11} & S_{r_12} & \ldots & S_{r_1 r_2} \cr
 	}
 	\ \ \ \ \ \ $ where $S_{ij}=A_{G_2}$ if $a_i$ and $a_j$ are adjacent in $G_1$ (and $0$ otherwise). As $G$ is the Tensor product of $G_1$ and $G_2$,  $A_G= A_{G_1 \times G_2}$ and hence the proof of forward part is complete.\\
 	
 	Conversely, let the adjacency matrix of $G$ be of the form 
 	
 	$\bordermatrix{ & &  &  & \cr
 		&	S_{11} & S_{12} & \ldots & S_{1r_2} \cr
 		&   S_{21} & S_{22} & \ldots & S_{2r_2} \cr
 		&   \vdots & \vdots & \vdots & \vdots \cr   
 		&   S_{r_11} & S_{r_12} & \ldots & S_{r_1 r_2} \cr
 	}
 	\ \ \ \ \ \ $  where $S_{ij}$ are either $B$ or $0$ (the zero matrix). Define $G_1=(V_1,E_1)$, where $V_1=\{a_1,a_2,\ldots a_{r_1}\}$ and $(a_i a_j) \in E_1$ if $B_{ij}=B$ and $G_2=(V_2,E_2)$ where $V_2=\{b_1,b_2,\ldots,b_{r_2}\}$ and $E_2$ be defined such that $A_{G_2}=B$. Then, as $n=r_1r_2$ and $A_G= A_{G_1 \times G_2}$, the converse holds and the proof is complete.
 \end{proof}
 
 \begin{Remark} \label{rem11}
It may be noted that for any graph $G$, $G=G_1\times G_2$ where $G_1=G$ and $G_2=(V_2,E_2)$ with $V_2=\{v_2\}$ is a singleton and $E_2$ consists of a self loop at $v_2$. Consequently, any graph $G$ can be written as a Tensor product of two graphs. However, such a factorization does not reduce the computational complexity of the shift space induced by $G$ and the reduction of the computational complexity indeed depends on the structure of the adjacency matrix of the underlying graph. Further, for given graphs $G_1=(V_1,E_1)$ and $G_2=(V_2,E_2)$, note that $(a_i, b_j) (a_k, b_l) \in \mathcal{E}(G_1  \scalebox{0.7}{$\square$} G_2)$  if and only if  $a_i a_k \in \mathcal{E}(G_1)$ ($a_i =a_k$) and $b_j =b_l$  ($b_j b_l \in \mathcal{E}(G_2)$). Further, if $\mathcal{V}(G_1) = \{a_1, a_2 \ldots a_n\}$, $\mathcal{V}(G_2) = \{b_1, b_2 \ldots b_m\}$ and the adjacency matrix of ($G_1  \scalebox{0.7}{$\square$} G_2)$ is indexed using natural order then for $a_i \neq a_j$, while $a_i a_j \notin \mathcal{E}(G_1)$  forces the corresponding $m\times m$ block to be zero matrix, $a_i a_j \in \mathcal{E}(G_1)$ forces the corresponding  $m\times m$ block to be an $m\times m$ identity matrix (for $a_i \neq a_j$). Finally, while $a_i=a_j$ then, $a_i a_i \notin \mathcal{E}(G_1)$ forces the corresponding $m\times m$ block to be $B$, and $a_i a_i \in \mathcal{E}(G_1)$ forces the corresponding block to be $B_{ii}= B*I$ (where $B$ is a binary matrix and $*$ is a matrix operation such that $(P*Q)_{ij}=max(P_{ij}, Q_{ij})$). Consequently, the adjacency matrix of the Cartesian product of graphs can be characterized using the corresponding block representation and we get the following corollary.
 \end{Remark}
 
 \begin{Cor}
 	For any given graph $G=(V,E)$ with $n$ vertices, $G$ can be expressed as a Cartesian product of two graphs if and only if there exists $r_1,r_2\in\mathbb{N}$ such that $n=r_1r_2$, a $r_2\times r_2$ matrix $B$ and a permutation $\{v_{\sigma(1)},v_{\sigma(2)}, \ldots, v_{\sigma(n)} \}$ of the vertex set $V$ such that adjacency matrix $A_G$ (with respect to $\{v_{\sigma(1)},v_{\sigma(2)},\ldots,v_{\sigma(n)} \}$) is of the form
 	
 	$\bordermatrix{ & &  &  & \cr
 		&	B_{11} & B_{12} & \ldots & B_{1r_2} \cr
 		&   B_{21} & B_{22} & \ldots & B_{2r_2} \cr
 		&   \vdots & \vdots & \vdots & \vdots \cr   
 		&   B_{r_11} & B_{r_12} & \ldots & B_{r_1 r_2} \cr
 	}
 	\ \ \ \ \ \ $  where 			$B_{ij} = \left\{%
 	\begin{array}{ll}
 		0 \text{~~or~~} I, & \hbox{$i\neq j$;} \\
 		B \text{~~or~~} (B*I)  & \hbox{$ i=j $;} \\
 	\end{array}%
 	\right.$
 \end{Cor}
 
 \begin{proof}
 	The proof follows from the discussion in Remark \ref{rem11}.
 \end{proof}

\bibliography{xbib}

\end{document}